\documentclass[
   colorlinks 
]{mpi2015-cscpreprint}

\usepackage[american]{babel}

\usepackage{graphicx}

\usepackage{amssymb}
\usepackage{amsthm}
\usepackage{mymacros}

\AtBeginDocument{
  \newtheorem{theorem}{Theorem}[section]
  \newtheorem{lemma}[theorem]{Lemma}
  \newtheorem{definition}[theorem]{Definition}
  \newtheorem{example}[theorem]{Example}
  \newtheorem{corollary}[theorem]{Corollary}
  \newtheorem{remark}[theorem]{Remark}
}

\usepackage[vlined,longend,linesnumbered,ruled]{algorithm2e}
\newcommand{\nnrm}[1]{{\left\vert\kern-0.25ex\left\vert\kern-0.25ex\left\vert #1 
		\right\vert\kern-0.25ex\right\vert\kern-0.25ex\right\vert}}

\begin{document}


\title{Computation of structured stability radii for Dissipative-Hamiltonian systems}

\author[$\dagger$]{Peter Benner}
\affil[$\dagger$]{Max Planck Institute for Dynamics of Complex Technical Systems, 39106 Magdeburg, Germany. \authorcr%
  \email{benner@mpi-magdeburg.mpg.de}, \orcid{0000-0003-3362-4103}}
  
\author[$\ddagger$]{Volker Mehrmann}
\affil[$\ddagger$]{Institute of Mathematics, TU Berlin, 10623, Berlin, Germany. \authorcr%
  \email{mehrmann@math.tu-berlin.de}, \orcid{0000-0001-5051-2870}}
  
\author[$\dagger$]{Anshul Prajapati}
\affil[$\dagger$]{Max Planck Institute for Dynamics of Complex Technical Systems, 39106 Magdeburg, Germany.\authorcr%
  \email{prajapati@mpi-magdeburg.mpg.de}, \orcid{0000-0002-7641-5304
}}
  
\author[$\mathsection$]{Punit Sharma}
\affil[$\mathsection$]{Department of Mathematics, Indian Institute of Technology Delhi, New Delhi-110016, India. \authorcr%
  \email{punit.sharma@maths.iitd.ac.in}, \orcid{0000-0003-4922-8877}\begin{center}\Large{
 In memory of Nicholas J. Higham.\vspace{-1em} }
\end{center} }


\shorttitle{STRUCTURED STABILITY RADII}
\shortauthor{P. BENNER ET AL. 
}
\shortdate{}


\keywords{dissipative Hamiltonian system, port-Hamiltonian system, structured stability radius, distance to instability, Rayleigh quotient, optimally robust representation}

\msc{93D20, 93D09, 65F15}

\abstract{%
    We study linear time-invariant Dissipative Hamiltonian (DH) systems  arising in energy-based modeling of dynamical systems.
    An advantage of DH systems is that they are always stable due to the structure of their coefficient matrices, and, under further weak conditions, even asymptotically stable. In this paper, we discuss the computation of the stability radii for a given asymptotically stable DH system; i.e., the smallest structured perturbation that puts a DH system on the boundary of the region of asymptotic stability, so that it has purely imaginary eigenvalues. We obtain explicit computable formulas for various structured stability radii. For this, the problem of computing stability radii is reformulated in terms of minimizing the Rayleigh quotient of a Hermitian matrix or the sum of two generalized Rayleigh quotients of Hermitian semidefinite matrices. This reformulation results in the problem of minimizing the largest eigenvalue of an eigenvector-dependent Hermitian matrix or minimizing the smallest eigenvalue of a Hermitian matrix which depends on the eigenvector.  It is also demonstrated (via numerical experiments) that, under structure-preserving perturbations, the asymptotic stability of a DH system is much more robust than under general perturbations, since the distance to instability is typically much larger when structure-preserving perturbations are considered.
   Finally, similar results are obtained for optimally robust representations of stable systems.
	}

 \novelty{The key contributions of our work are the following:
	\begin{itemize}
		\item We derive computable formulas for structured stability radii of DH systems under various structure-preserving perturbations, in terms of optimizing a function involving two generalized Rayleigh quotients or optimizing the largest eigenvalue of an eigenvector-dependent Hermitian matrix.
		\item Motivated by~\cite{LuPSB25}, a nonlinear eigenvalue problem where the coefficient matrix depends on the eigenvector, is obtained for the optimization involving two generalized Rayleigh quotients.
		\item Several bounds are obtained for various structured stability radii.
		\item Numerical experiments are presented to illustrate the robustness of DH systems under structure-preserving perturbations.
		\item In particular, the robustness under structured perturbations for optimally robust DH representations is considered.
	\end{itemize}
}

\maketitle



\section{Introduction}%
\label{sec:intro}

In this paper, we consider  the (asymptotic) stability analysis and in particular the computation of stability radii for the class of \emph{linear time-invariant port-Hamiltonian (pH)  systems} of the form 
\begin{align*}
    \dot{x}&=(J-R)Qx + (G-P)u,\\
    y&=(G+P)^*Qx + (S+N)u,
\end{align*}
where $x,u$, and $y$ are the state, input, and output of the system. The function $x\mapsto x^*Qx$, with $Q^*=Q$
positive definite, describes the \emph{stored energy} of the system and is called the \emph{Hamiltonian} of the system. The coefficient matrices $J, R, Q\in \C^{n,n}, G , P\in \C^{n,m}, S, N \in\C^{m,m}$, where $\C^{n,n}$ is the set of $n\times n$ complex matrices, satisfy 
\begin{equation*}
    \mat{cc} J & G\\-G^* & N \rix^* = - \mat{cc} J & G\\-G^* & N \rix, \qquad W=W^*=\mat{cc} R & P \\ P^* & S \rix \geq 0,
\end{equation*}
where $W\geq 0$  ($W>0$) denotes that the matrix $W$ is positive semidefinite (definite).

PH systems possess numerous significant geometric and algebraic properties that are nicely encoded in their representation; see~\cite{DalV98,JacZ12,OrtVME02,MehU23}. In this paper, we focus on the property that PH systems are stable, i.e., all eigenvalues of the system matrix $A = (J - R)Q$ are in the closed left half of the complex plane, and all eigenvalues located on the imaginary axis are semisimple, see \cite{MehMS16,MehMW18}. For the stability analysis, the port matrices can be ignored, resulting in a \emph{Dissipative Hamiltonian (DH)} system represented by
\begin{equation}
    \label{DHsystem} \dot{x} = (J - R)Qx,
\end{equation}
where $J$ is skew-Hermitian, $R\geq 0$, $Q>0$.

If the product in \eqref{DHsystem} is multiplied out to create a matrix $A=(J-R)Q$ and the DH structure  is ignored, then the stability of the system is no longer evident from the structure of the coefficients. In this case, to see  if the system is stable, one can either compute the eigenvalues or use the Lyapunov theorem~\cite{LanT85}. If $A$ has purely imaginary eigenvalues, then arbitrarily small perturbations (such as data or roundoff errors) can cause eigenvalues to move to the right half complex plane. This is particularly the case for linear systems that emerge from the linearization of nonlinear systems near stationary reference solutions~\cite{Cam95}, from data-driven realizations (see, e.g.,~\cite{Ant05,MayA07}), or from classical finite element modeling~\cite{GraMQSV16}. Since in practice, the system model experiences perturbations, stability can only be assured when the system maintains a sufficient distance from instability; see~\cite{HinP86,HinP05}. Computing the distance to instability~\cite{Bye88,FreS11,HeW98,Van84} is an optimization problem and usually again subject to perturbations.

The situation is different for DH systems which are automatically stable under perturbations, as long as the DH structure is preserved. However, DH systems are not necessarily asymptotically stable, i.e., they may have purely imaginary eigenvalues. Therefore, it is important to know whether a DH system is simply stable or even asymptotically stable. Knowing whether the system is robustly asymptotically stable is even more important, meaning that small (structured) perturbations maintain its asymptotic stability. The latter again requires that the system has a reasonable distance from a DH system with purely imaginary eigenvalues. To study this question is an important topic in many applications, in particular, in power systems and circuit simulation (see, for example,~\cite{Mar86,MarL90,MarPR07,RomM09}) and multibody systems (see, for example,~\cite{GraMQSV16,Sch90,Ves11}).

\begin{example}\label{example:brakesqueal}\emph{
    Large-scale second-order differential equations of the form
    \begin{equation*}
        M\ddot{x}+(D+G)\dot{x}+(K+N)x = f,
    \end{equation*}
    arise in the finite element analysis of disk brake squeal~\cite{GraMQSV16}, where $M= M^* >0$ is the mass matrix, $D= D^* \geq 0$ represents material and friction-induced damping, $G=-G^*$ represents gyroscopic effects, $K= K^* >0$ represents stiffness, and $N$ is  non-symmetric and accounts for circulatory effects. 
    An appropriate first order formulation is associated with the pencil $\lambda I+ (J-R)Q$, where
    \begin{equation*}\label{eq:brakesqueal}
        J:=\mat{cc} G & (K+\frac{1}{2}N)\\ -(K+\frac{1}{2}N^*) & 0\rix,\qquad R:= \mat{cc} D & \frac{1}{2}N \\ \frac{1}{2}N^* & 0 \rix,\qquad Q:= \mat{cc} M & 0 \\ 0 &K \rix^{-1},
    \end{equation*}
    and $I$ denotes the identity matrix.}
\end{example}

Brake squeal is associated with eigenvalues in the right half-plane. In the absence of circulatory effects, i.e., when $N$ vanishes, then the system is automatically stable, since it is a DH system. One can view the matrix $N$ as a (low-rank) perturbation of a DH system, since in the industrial examples considered in~\cite{GraMQSV16}, the matrix $N$ has a rank of order $2000$ and the size of the system is of the order $1$ million. It is obvious that for $N \neq 0$ the pencil $\lambda I+ (J-R)Q$ is missing one of the essential properties of a DH system, because the matrix $R$ is then indefinite, and thus the system may be unstable, which is the reason for the squeal. To analyze properties of the system~\eqref{DHsystem} when this happens is one of the motivations for our work.

It has been observed already a long time ago that the perturbation bounds under structured perturbation differ significantly from those under unstructured perturbation; see e.g. the landmark paper~\cite{HigH92}. Our work is following this direction of research.

In recent years, there has been an increasing interest in the investigation of the stability radii of dissipative and port-Hamiltonian systems under structure preserving perturbations; see~\cite{AliMM20,BagGS21,MehMS16,MehMS17,MehMW21}. In~\cite{MehMS16}, the stability radii of DH systems were examined by analyzing structure-preserving perturbations to $J$, $R$, and $Q$ individually, resulting in computable formulas. In~\cite{AliMM20}, stability radii approximations for large-scale port-Hamiltonian systems were derived, focusing solely on perturbations in $R$. Additionally, in~\cite{BagGS21}, perturbations to $J$ and $R$ were considered, and a lower bound was derived for the structured stability radii.

In this paper, we revisit the stability radii of DH systems under various structure-preserving perturbations to both $J$ and $R$. We obtain new explicit formulas in terms of optimizing a function involving two generalized Rayleigh quotients or optimizing the largest eigenvalue of an eigenvector-dependent Hermitian matrix. We also study the stability radii of recently developed robust representations of DH systems. 

The paper is organized as follows.
In Section~\ref{sec:prelim}, we define various stability radii under different structure-preserving perturbations. We then recall some preliminary results that will be used to characterize these structured distances. In Section~\ref{sec:radius}, we derive computable formulas for three different stability radii under structure-preserving perturbations to both $J$ and $R$, i.e., skew-Hermitian perturbations  of the form $J+\Delta_J$ and three different types of structured perturbations of the form $R+\Delta_R$, first when $\Delta_R$ is Hermitian negative semidefinite such that $R+\Delta_R$ is Hermitian positive semidefinite, second when $\Delta_R$ is Hermitian such that $R+\Delta_R$ is Hermitian positive semidefinite, and third when $\Delta_R$ is Hermitian. 
Section~\ref{sec:minSRQ2} briefly discusses the minimization of a function involving two generalized Rayleigh quotients via the minimization of a rational function over a joint numerical range. We use this function in computing the structured stability radii under skew-Hermitian perturbations to $J$ and Hermitian negative semidefinite perturbations to $R$ such that $R+\Delta_R \geq 0$. Section~\ref{sec:numerical} presents numerical experiments to demonstrate the results that we have obtained, particularly highlighting that the stability distances under structure-preserving perturbations can substantially differ from those under unstructured  perturbations.

\section{Preliminaries}\label{sec:prelim}
In the following, $\|\cdot\|$ denotes the spectral norm of a vector or a matrix, while $\|\cdot\|_F$ denotes the Frobenius norm of a matrix. By $\Lambda(A)$, we denote the spectrum of a matrix $A \in \C^{n,n}$. 
We use the notation $A \geq 0$ and $A \leq 0$ if $A \in \C^{n,n}$ is Hermitian and positive or  negative semidefinite, respectively, and $A > 0$ if $A$ is Hermitian positive definite.
We denote the identity matrix of size $n$ by $I_n$. For a complex number $z$, $\Im(z)$ and $\Re(z)$, respectively, denote the imaginary and the real part of $z$. We denote by $\sigma_{\text{min}}(A)$  the smallest singular value of a matrix $A$. If $R$ is Hermitian, then $\lambda_{\text{max}}(R)$ and $\lambda_{\text{min}}(R)$ denote its largest and smallest eigenvalue, respectively. 
We will frequently use generalized Rayleigh quotients $\rho(x) := \frac{x^*H_1x}{x^*H_2x}$, with $x\in\C^n\setminus\{0\}$, for positive semidefinite $H_1$ and $H_2\in\C^{n,n}$, 
where we define that
\begin{equation}\label{eq:defrq}
	\rho(x)= \frac{x^*H_1x}{x^*H_2x} :=0, \quad\text{if both $x^*H_1x = 0$ and $x^*H_2x= 0$.}
\end{equation}
By ~\eqref{eq:defrq}, the function $\rho(x)$ is well-defined and lower semi-continuous, i.e., $\liminf_{y\to x}\rho(y)\geq \rho(x)$, for all $x\in\C^n\setminus\{0\}$, so that we can properly define optimization problems involving such Rayleigh quotients.

In the following, we consider different perturbations in the coefficient matrices $J$ and $R$ of a DH system of the form~\eqref{DHsystem}. These take the form
\begin{equation}
    \tilde A=(J+\Delta_J - (R+\Delta_R))Q,
\end{equation}
i.e.\ we do not perform perturbations in the energy matrix $Q$. This is a reasonable assumption as the energy is often provided analytically and is not subject to perturbations in such situations.
In order to measure these perturbations in $J$ and $R$, we consider the following norm on the space $(\C^{n,n})^2$. For a given tuple $(\Delta_J,\Delta_R)\in (\C^{n,n})^2$, this is
\begin{equation}
    \nnrm{(\Delta_J,\Delta_R)}=\sqrt{\|\Delta_J\|^2+\|\Delta_R\|^2}.
\end{equation}
Note that we could also {measure the perturbations in a matrix norm induced by an energy norm weighted by $Q$, i.e. $\|x\|_Q=\sqrt{x^HQx}$.}

For complex unstructured linear systems that are asymptotically stable, the smallest norm of a perturbation that moves an eigenvalue to the imaginary axis is called the \emph{(complex) stability radius}, since arbitrarily small perturbations can then move an eigenvalue to the right half-plane and thus make the system unstable. For real systems, there is also the \emph{real stability radius}, which refers to perturbations that are constrained to be real; see~\cite{HinP90}.

In the case of DH systems, if we preserve the DH structure while performing perturbations, then we may loose asymptotic stability, but the system remains stable. Despite this, the term stability radius has been used in literature; see~\cite{MehMS16}. We follow the terminology in~\cite{MehMS16,MehMS17}, and define the stability radii for DH systems of the form~\eqref{DHsystem} as follows.

\begin{definition}\label{def:radii}
    Consider a DH system of the form~\eqref{DHsystem}. Then the \emph{unstructured stability radius} with respect to arbitrary perturbations in $J$ and $R$ is defined by 
    \begin{equation}
        r(J,R):=\inf \left\{ \nnrm{(\Delta_J,\Delta_R)}~:~\Delta_J,\Delta_R\in \C^{n,n},~ \Lambda((J+\Delta_J - (R+\Delta_R))Q) \cap i\R \neq \emptyset \right\}.
    \end{equation} 
    For structure-preserving perturbations in $J$ and $R$, we consider the following cases.
    \begin{enumerate}
        \item The \emph{stability radius $r^{S_d}(J,R)$} with respect to skew-Hermitian perturbations in $J$ and Hermitian negative semidefinite perturbations in $R$ from the perturbation set
        \begin{equation}\label{eq:setSd}
            S_d(J,R)=\{(\Delta_J,\Delta_R)\in (\C^{n,n})^2~:~ \Delta_J^*=-\Delta_J,~\Delta_R^*=\Delta_R\leq 0, ~ R+\Delta_R\geq0\}
        \end{equation}
        is defined by 
        \begin{equation}\label{eq:radSd}
            r^{S_d}(J,R):=\inf \left\{ \nnrm{(\Delta_J,\Delta_R)} :(\Delta_J,\Delta_R)\in S_d(J,R),~ \Lambda((J+\Delta_J - (R+\Delta_R))Q) \cap i\R \neq \emptyset \right\}.
        \end{equation}
        \item The \emph{stability radius $r^{S_i}(J,R)$} with respect to skew-Hermitian perturbations in $J$ and Hermitian but not necessarily semidefinite perturbations to $R$ from the perturbation set
        \begin{equation}\label{eq:setSi}
            S_i(J,R)=\{(\Delta_J,\Delta_R)\in (\C^{n,n})^2 ~:~ \Delta_J^*=-\Delta_J,~\Delta_R^*=\Delta_R, ~ R+\Delta_R\geq0\}
        \end{equation}
        is defined by
        \begin{equation}\label{eq:radSi}
            r^{S_i}(J,R):=\inf \left\{ \nnrm{(\Delta_J,\Delta_R)} : (\Delta_J,\Delta_R)\in S_i(J,R),~ \Lambda((J+\Delta_J - (R+\Delta_R))Q) \cap i\R \neq \emptyset \right\}.
        \end{equation}
        \item The \emph{eigenvalue backward error $\eta^{S}(J,R,\lambda)$, $\lambda\in \C$} and the \emph{stability radius $r^{S}(J,R)$} with respect to skew-Hermitian perturbations in $J$ and Hermitian perturbations in $R$ from the perturbation set
        \begin{equation}\label{eq:setS}
            S(J,R)=\{(\Delta_J,\Delta_R)\in (\C^{n,n})^2 ~:~ \Delta_J^*=-\Delta_J,~\Delta_R^*=\Delta_R\}
        \end{equation}
        are respectively defined by
        \begin{equation}\label{eq:backerror}
            \eta^{S}(J,R,\lambda):=\inf\left\{\nnrm{(\Delta_J,\Delta_R)}~:~(\Delta_J,\Delta_R)\in S(J,R),~ \lambda \in \Lambda((J+\Delta_J - (R+\Delta_R))Q)\right\}
        \end{equation}
        and
        \begin{equation}\label{eq:radS}
            r^{S}(J,R):=\inf \left\{\nnrm{(\Delta_J,\Delta_R)}~:~(\Delta_J,\Delta_R)\in S(J,R),~ \Lambda((J+\Delta_J - (R+\Delta_R))Q) \cap i\R \neq \emptyset \right\}.
        \end{equation}
    \end{enumerate}
\end{definition}

{Note that the sets $S_d(J,R)$ and $S_i(J,R)$ depend implicitly on the matrix $R$ through the constraint $R+\Delta_R \geq 0$, which determines the admissible perturbations $\Delta_R$. While the perturbation structure itself does not explicitly depend on $J$, we use the notation $S_d(J,R)$ and $S_i(J,R)$ to emphasize that these sets correspond to structured perturbations to the pair $(J,R)$. For the sake of uniformity we use the same notation for the set $S(J,R)$.

The additional constraint $\Delta_R\leq 0$ in $S_d(J,R)$ is motivated by the observation that, under the requirement $R+\Delta_R\geq 0$, the optimal perturbation attaining the distance to instability is generally of rank two. Imposing $\Delta_R\leq 0$ allows for a rank-one characterization of the distance to instability while still preserving the dissipative structure $R+\Delta_R\geq 0$.

In Definition~\ref{def:radii}, if the perturbation in $R$ is further restricted to be of rank one, then we denote this by adding an index $1$, i.e., we write $r_1$ for the corresponding radius.}

The characterization of the stability radius $r(J,R)$ can be obtained by slightly modifying the general approach~\cite{HinP86}, see also~\cite{BagGS21} for an independent derivation.

{
\begin{theorem}\label{thm:unstr}
    Consider an asymptotically stable DH system of the form~\eqref{DHsystem}. Then the unstructured stability radius $r(J,R)$ is finite and is given by
    \begin{equation}
        r(J,R)=\frac{1}{\sqrt{2}}\inf_{w\in \R}\frac{1}{\|G(\omega)\|}.
    \end{equation} 
\end{theorem}}

We will discuss in detail the stability radii $r^{S_d}(J,R)$, $r^{S_i}(J,R)$, and $r^{S}(J,R)$ defined in Definition~\ref{def:radii} 
and will compare them with the stability radius $r(J,R)$. 
In order to do this, we make use of the {following mapping results, which summarize the results from} \cite{BorKMS14,MehMS16} adapted to our setting.
\begin{lemma}\label{map:herm}
    Let $x \in \C^n\setminus\{0\}$ and $y\in \C^{n}$. Then 
    \begin{itemize}
        \item [a)] there exists a Hermitian matrix $H \in \C^{n,n}$ such that $Hx = y$ if and only if $\Im{(x^*y)} = 0$ and we have
            \[
        \min\big\{\|H\|\; : H\in \C^{n,n},\; H^*=H, \; Hx=y\big\} = \frac{\|y\|}{\|x\|},
        \] 
        and the minimum is attained by
        \begin{equation}\label{def:hatH}
            \hat H_{(x,y)} := \frac{\|y\|}{\|x\|}\mat{cc}\frac{y}{\|y\|}&\frac{x}{\|x\|} \rix \mat{cc}\frac{y^*x}{\|x\|\|y\|}&1\\1&\frac{x^*y}{\|x\|\|y\|} \rix^{-1} \mat{cc}\frac{y}{\|y\|}&\frac{x}{\|x\|} \rix^*
        \end{equation}
        if $x$ and $y$ are linearly independent and by $\hat H_{(x,y)}:=\frac{yx^*}{x^*x}$, otherwise;
        \item [b)] there exists a skew-Hermitian matrix $S \in \C^{n,n}$ such that $Sx = y$ if and only if $\Re{(x^*y)} = 0$ and we have
        \[
        \min\big\{\|S\|\; :\; S ^*=-S, \; Sx=y\big\} = \frac{\|y\|}{\|x\|},
        \] 
        and the minimum is attained by $\hat S:=-i\hat H_{(x,iy)}$, where $\hat{H}$ is defined in~\eqref{def:hatH}.
    \end{itemize}
\end{lemma}
%
    %
{
\begin{lemma}\label{map:def}
	Let $x\in \C^n\setminus\{0\},y \in \C^{n}$. Consider the set $S=\{\Delta\in \C^{n,n}~:~\Delta^*=\Delta\leq 0, \;\Delta x=y\}.$
    \begin{enumerate}
        \item If $y\neq 0$, then the set $S$ is non-empty if and only if {{$x^*y \in \R$ such that $x^*y < 0$}}. If the latter condition is satisfied, then
    	\[
    	\min\left\{\|\Delta\|:~\Delta \preceq 0,\, \Delta x=y\right\}=\frac{{\|y\|}^2}{|x^*y|}
    	\]
    	and the minimum is attained by the rank one matrix $\hat \Delta_{(x,y)}=\frac{1}{x^*y}yy^*$.
        \item If $y=0$, then $\hat \Delta=0$ is the minimal norm matrix from $S$.
    \end{enumerate}
\end{lemma}}
%

In the following section we present structured perturbation results.
\section{Stability radii of DH systems under structure preserving perturbations to the matrices $J$ and $R$.}\label{sec:radius}
In this section, we present analytic results for the  three stability radii defined in Definition~\ref{def:radii} with respect to various structured perturbations to $J$ and $R$. 

\subsection{The structured stability radius \texorpdfstring{$r^{S_d}(J,R)$}{TEXT}}\label{sec:def}
In this subsection, we derive a formula for the stability radius $r^{S_d}(J,R)$ defined in~\eqref{eq:radSd} under perturbations to $J$ and $R$ from the set $S_d(J,R)$ defined in~\eqref{eq:setSd}.
%
%
\begin{theorem}\label{thm:def}
    Consider an asymptotically stable DH system of the form~\eqref{DHsystem}. Then
    \begin{equation}\label{eq:def}
        (r^{S_d}(J,R))^2= \inf_{\omega \in \R} \inf_{x\in \C^{n}\setminus \{0\}} \left\{\left(\frac{x^*QR^2Qx}{x^*QRQx} \right)^2+\frac{x^*(i\omega I_n -JQ)^*(i\omega I_n -JQ)x}{x^*Q^2x} \right\}.
    \end{equation}
    {
    Furthermore, if the infimum in~\eqref{eq:def} is attained at $\hat \omega$ and $\hat x$, then for $\hat \Delta_J=-i \hat H_{(Q\hat x, i(i\hat \omega I_n-JQ)\hat x)}$, where $\hat H$ is defined in~\eqref{def:hatH} and $\hat \Delta_R=\frac{-1}{\hat x^*QRQ\hat x}(RQ\hat x)(RQ\hat x)^*$, we have 
    \begin{equation*}
        r^{S_d}(J,R)=\sqrt{\|\hat \Delta_J\|^2+\|\hat \Delta_R\|^2}.
    \end{equation*}}
\end{theorem}
\begin{proof}
    By Definition~\ref{def:radii}, we have
    \begin{equation}
        r^{S_d}(J,R)=\inf \left\{ \nnrm{(\Delta_J,\Delta_R)}~:~(\Delta_J,\Delta_R)\in S_d(J,R),~ \Lambda((J+\Delta_J - (R+\Delta_R))Q) \cap i\R \neq \emptyset \right\}.
    \end{equation}
    Since for $(\Delta_J,\Delta_R)\in S_d(J,R)$ the perturbed DH system matrix $(J+\Delta_J - (R+\Delta_R))Q$ still has the DH structure, by using~\cite[Lemma 3.1]{MehMS16}, we obtain
    \begin{align}
        &r^{S_d}(J,R) \nonumber\\
        &=\inf \{ \nnrm{(\Delta_J,\Delta_R)}:(\Delta_J,\Delta_R)\in S_d(J,R),~ (R+\Delta_R)Qx=0 \text{ for some eigenvector $x$ of } (J+\Delta_J)Q \}\nonumber\\
        &=\inf\{ \nnrm{(\Delta_J,\Delta_R)}:(\Delta_J,\Delta_R)\in S_d(J,R),~ (R+\Delta_R)Qx=0 \text{ for some $x\in \C^n\setminus\{0\}$ satisfying } \nonumber\\ &\hspace{5cm}  (J+\Delta_J)Qx= i\omega x, \omega\in \R\}\nonumber\\
        &= \inf_{\omega\in \R}\inf_{x\in\C^{n}\setminus \{0\}} \Big(\inf \{ \nnrm{(\Delta_J,\Delta_R)}:(\Delta_J,\Delta_R)\in S_d(J,R),~ (R+\Delta_R)Qx=0, (J+\Delta_J)Qx= i\omega x\} \Big)\nonumber\\
        &=\inf_{\omega \in \R } \vartheta^{S_d}_{\omega} \label{eq:def11},
    \end{align}
    %
{where for a given scalar $\omega$, we have}
    \begin{align}\label{eq:def12}
        \vartheta^{S_d}_{\omega}&:= \inf_{x\in\C^{n}\setminus \{0\}} \Big(\inf \{ \nnrm{(\Delta_J,\Delta_R)}:(\Delta_J,\Delta_R)\in S_d(J,R),~ (R+\Delta_R)Qx=0, \nonumber \\ & \hspace{6cm}(J+\Delta_J)Qx= i\omega x\} \Big).
    \end{align}
     For the inner optimization in~\eqref{eq:def12}, we will show now that the minimal value can be expressed as the sum of two generalized Rayleigh quotients, which depend on the variables $x$ and $\omega$. For this, we first solve two mapping problems.
     
     For given $x\in \C^n\setminus\{0\}$ and $\omega\in \mathbb R$, determine the minimum norm solution $(\Delta_J,\Delta_R)$ such that $\Delta_RQx=-RQx$ and $\Delta_JQx=(i\omega I_n - JQ)x$. Using Lemma~\ref{map:herm}, there exists a skew-Hermitian $\Delta_J$ satisfying $\Delta_JQx=(i\omega I_n - JQ)x$ if and only if $\Re(x^*Q(i\omega I_n - JQ)x)=0$, which trivially holds because of the structure of $J$ and $Q$. The minimal norm among all such mappings $\Delta_j$ is given by
    \begin{equation}\label{norm:jdef}
        \|\hat \Delta_J\|=\frac{\|(i\omega I_n - JQ)x\|}{\|Qx\|},
    \end{equation}
    and is attained by $\hat \Delta_J=-i \hat H_{(Qx, i(i\omega I_n-JQ)x)}$, where $\hat H$ is defined in~\eqref{def:hatH}. 
    
    For the other constraint $\Delta_RQx=-RQx$, note that $Qx\neq 0$ as $x\neq 0$ and $Q$ is a positive definite matrix. If $RQx=0$ (this case may arise if $R$ is singular), then from Lemma~\ref{map:def}, $\Delta_R=0$ is the solution that satisfies~\eqref{norm:rdef} by defining the fraction $\frac{0}{0}$ to have value $0$. If $RQx\neq 0$, then from Lemma~\ref{map:def}, there exists $\Delta_R\leq 0$ satisfying $\Delta_RQx=-RQx$ if and only if $-x^*QRQx<0$, which holds as $R\geq 0$, and the minimum norm solution 
    is given by
     \begin{equation}\label{norm:rdef}
        \|\hat \Delta_R\|=\frac{\|RQx\|^2}{x^*QRQx},
     \end{equation}
     which is attained by $\hat \Delta_R=\frac{-1}{x^*QRQx}(RQx)(RQx)^*$. {Note that
    \[
    (\hat \Delta_J,\hat \Delta_R)\in
    \{(\Delta_J,\Delta_R)~:~ \Delta_J^*=-\Delta_J,\ \Delta_R\leq 0\}
    \supseteq S_d(J,R).
    \]
    Hence, it is enough to show that $R+\hat \Delta_R\geq 0$. Then $(\hat \Delta_J,\hat \Delta_R)\in S_d(J, R)$, and since it is optimal over the larger set, it is also optimal over $S_d(J, R)$. Now in view of {\cite[Lemma 4.1]{MehMS16}} it follows that for the optimal perturbation $\hat \Delta_R$ we have $R+\hat \Delta_R\geq0$ and thus $(\hat \Delta_J,\hat \Delta_R)\in S_d(J, R)$. 
    Using this in~\eqref{eq:def12}, we have
    \begin{equation}\label{eq:tempte1}
        (\vartheta^{S_d}_{\omega})^2=\inf_{x\in\C^{n}\setminus \{0\}}
    \frac{\|RQx\|^4}{(x^*QRQx)^2} + \frac{\|(i\omega I_n - JQ)x\|^2}{\|Qx\|^2}.    
    \end{equation}
    Thus result follows from~\eqref{eq:tempte1} and~\eqref{eq:def11}.
    } 
\end{proof}

{
\begin{remark}
It follows from Theorem~\ref{thm:def} that the minimal perturbation $\hat \Delta_R$ in $R$ that attains the stability radius $r^{\mathcal S_d}(J,R)$ can be chosen to be of rank one. On the other hand, any Hermitian rank one perturbation $\Delta_R$ in $R$ and skew-Hermitian perturbation $\Delta_J$ in $J$ of $(J-R)Q$ such that 
    $((J+\Delta_J)-(R+\Delta_R))Q$ has an eigenvalue on the imaginary axis implies that $\Delta_R$ must necessarily be negative definite and the norm
    $\sqrt{{\|\Delta_J\|}^2+{\|\Delta_R\|}^2}$ must at least be $r^{{\mathcal S}_d}(J,R)$. Consequently, we have 
    \[
    r^{{\mathcal S}_d}(J,R)=r_1^{{\mathcal S}_d}(J,R)=r_1^{{\mathcal S}_i}(J,R).
    \]
    Thus, by considering negative semidefinite perturbation matrices $\Delta_R$ shows that the minimal perturbation that moves an eigenvalue to the imaginary axis can be chosen to be of rank one.
\end{remark}}

Note that the objective function 
in~\eqref{eq:def} depends on two generalized Rayleigh quotients, i.e.,  on four Hermitian positive semidefinite matrices. However, with the change of coordinates $y:=Qx$, the objective function can be represented in the variable $y$, which gives
\begin{equation*}
	\phi(y)=\left( \frac{y^*R^2y}{y^*Ry} \right)^2 + \frac{y^*Q^{-1}(i\omega I_n - JQ)^*(i\omega I_n - JQ)Q^{-1}y}{y^*y}.
\end{equation*}
One advantage of writing the objective function in this form is that we can further reformulate this optimization into the minimization of a rational function over the joint numerical range of three Hermitian matrices; we discuss this in more detail in Section~\ref{sec:minSRQ2}. This reformulation is actually the formulation that is used in DH differential-algebraic equations, see \cite{MehU23} for a detailed survey.

We directly obtain the following bounds for $r^{S_d}(J,R)$.
%
%
\begin{equation*}
	r(J,R)\leq r^{S}(J,R)\leq r^{S_i}(J,R) \leq r^{S_d}(J,R),
\end{equation*}
which automatically gives lower bounds to $r^{S_d}(J,R)$. Further detailed bounds for $r^{S_i}(J,R)$ and $r^{S}(J,R)$ are discussed in Sections~\ref{sec:idef} and~\ref{sec:Herm}, respectively. 

Another lower bound for $r^{S_d}(J,R)$, which is a direct consequence of Theorem~\ref{thm:def} is presented  in the following corollary, which also gives an upper bound for  $r^{S_d}(J,R)$.
%
\begin{corollary}\label{cor:def}
    Consider an asymptotically stable DH system of the form~\eqref{DHsystem}. Let  $V^*RV=D$ be the spectral decomposition of $R$, where $V\in \C^{n,n}$ is unitary and $D=\text{diag}(d_1,d_2,\ldots,d_n)$ is such that $d_1\geq d_2\geq \ldots\geq d_n\geq 0$. Then
    \begin{equation}\label{eq:defbound}
        d_n^2 + \inf_{\omega \in \R} \lambda_{\min}(Q^{-1}(i\omega I_n -JQ)^*(i\omega I_n -JQ)Q^{-1}) \leq r^{S_d}(J,R)^2\leq d_n^2 + \inf_{\omega\in\R} \frac{\|(i\omega I_n-JQ)\hat x\|^2}{\|Q\hat x\|^2},
    \end{equation}
    where $\hat x=Q^{-1}Ve_n$, with $e_n$ being the nth column of the $n\times n$ identity matrix.
\end{corollary}
\begin{proof}
    For the lower bound, by Theorem~\ref{thm:def} and the change of coordinates $y=Qx$, we have
    \begin{align}
        r^{S_d}(J,R) & = \inf_{\omega \in \R} \inf_{x\in \C^{n}\setminus \{0\}} \left\{\left(\frac{x^*QR^2Qx}{x^*QRQx} \right)^2+\frac{x^*(i\omega I_n -JQ)^*(i\omega I_n -JQ)x}{x^*Q^2x} \right\} \nonumber \\
        & \geq \inf_{\omega \in \R} \inf_{x\in \C^{n}\setminus \{0\}} \left(\frac{x^*QR^2Qx}{x^*QRQx} \right)^2 + \inf_{\omega \in \R} \inf_{x\in \C^{n}\setminus \{0\}} \frac{x^*(i\omega I_n -JQ)^*(i\omega I_n -JQ)x}{x^*Q^2x}\nonumber \\ 
        & = \inf_{y\in \C^{n}\setminus \{0\}} \left( \frac{y^*R^2y}{y^*Ry} \right)^2 + \inf_{\omega \in \R} \inf_{y\in \C^{n}\setminus \{0\}}\frac{y^*Q^{-1}(i\omega I_n - JQ)^*(i\omega I_n - JQ)Q^{-1}y}{y^*y} \label{cor:def:eq1},\\
        & = d_n^2 + \inf_{\omega \in \R} \lambda_{\text{min}}(Q^{-1}(i\omega I_n -JQ)^*(i\omega I_n -JQ)Q^{-1}).\label{cor:def:eq2}
    \end{align}
    {Note that~\eqref{cor:def:eq2} above follows from~\eqref{cor:def:eq1}, as the infimum in the first part} of~\eqref{cor:def:eq1} has the value $d_n^2$, where $d_n$ is the smallest eigenvalue of  $R$. Indeed, if $R>0$ then the infimum is given by the {minimal eigenvalue of the pair $(R^2,R)$ (equivalently of the pencil $R^2-\lambda R$), i.e. $\lambda_{\text{min}}(R^2,R)= \lambda_{\text{min}}(R)=d_n$.}  However, if $R\geq 0$ is singular, then $\frac{y^*R^2y}{y^*Ry} \geq 0 = d_n^2$, and the infimum is attained by the eigenvector $\hat y$ corresponding to the smallest eigenvalue of $R$ by letting the fraction $\frac{0}{0}$ have the value 0.
    %
	
    From the proof of Theorem~\ref{thm:def}, we have
    \begin{align*}\label{eq:defbound1}
        &r^{S_d}(J,R)^2= \inf_{\omega\in \R}\Big(\inf_{x\in\C^{n}} \inf \{ \nnrm{(\Delta_J,\Delta_R)}^2:(\Delta_J,\Delta_R)\in S_d(J,R),~ (R+\Delta_R)Qx=0 \text{  and } \nonumber \\ &\hspace{9cm}  (J+\Delta_J)Qx= i\omega x\} \Big).
    \end{align*}
    %
    %
    {
    To derive an upper bound, we avoid solving the minimization problem over the variable $x$ and instead fix $\hat x = Q^{-1}Ve_n.$ For this choice $x=\hat x$, we determine the optimal (minimum-norm) perturbations $\hat\Delta_R$ and $\hat\Delta_J$ satisfying
    \[
    (R+\hat \Delta_R)Q\hat x=0, \qquad (J+\hat \Delta_J)Q\hat x=i\omega \hat x.
    \]
    This yields an upper bound for the distance rather than the exact minimizer of the optimization problem.
    
    It is straightforward to verify that the minimum-norm perturbation $\hat\Delta_R$ satisfying 
    \[
    \hat\Delta_R\le 0,\ R+\hat\Delta_R\ge 0,
    \ (R+\hat\Delta_R)Q\hat x=0,
    \]
    is given by $\hat\Delta_R=-d_nI_n.$ Furthermore, by Lemma~\ref{map:herm}, there exists a skew-Hermitian matrix $\Delta_J$ satisfying $(J+\Delta_J)Q\hat x=i\omega \hat x$ if and only if ${\Re}(\hat x^*Q(i\omega I_n-JQ)\hat x)=0,$ which is automatically satisfied. The corresponding minimum-norm perturbation $\hat\Delta_J$ satisfies
    \[
    \|\hat\Delta_J\| = \frac{\|(i\omega I_n-JQ)\hat x\|}{\|Q\hat x\|}.
    \]
    Using these minimum-norm perturbations, we obtain
    \[
    r^{S_d}(J,R)^2 \leq \inf_{\omega\in\R} \|\hat\Delta_R\|^2+\|\hat\Delta_J\|^2  \leq d_n^2 + \inf_{\omega\in\R} \frac{\|(i\omega I_n-JQ)\hat x\|^2}{\|Q\hat x\|^2}
    \]
    which completes the proof.
    }
\end{proof}

\begin{remark}\emph{
    Note that taking any particular vector $x\in \C^n\setminus\{0\}$ in~\eqref{eq:def} will give an upper bound to $r^{S_d}(J,R)$. The upper bound given in Corollary~\ref{cor:def} is one such bound with the vector $x=Q^{-1}v$, where $v$ is the eigenvector corresponding to the smallest eigenvalue of $R$.
    }
\end{remark}


\subsection{The structured stability radius \texorpdfstring{$r^{S_i}(J,R)$}{TEXT}}\label{sec:idef}
In this subsection, we discuss the stability radius $r^{S_i}(J,R)$ defined in~\eqref{eq:radSi} under perturbations to $J$ and $R$ from the set $S_i(J,R)$ with Hermitian but possibly indefinite perturbations in $R$ of the form $R+\Delta_R\geq 0$. 
%
%
\begin{theorem}\label{thm:idef}
    Consider an asymptotically stable DH system of the form~\eqref{DHsystem}. For $\omega\in \R$, define $G(\omega):=\mat{cc} RQ \\ (i\omega I_n -JQ) \rix$.
    {
    Then
    %
        \begin{equation}\label{eq:idef01}
            r^{S_i}(J,R) = \inf_{\omega\in \R} \sigma_{\min} \left( G(\omega) Q^{-1} \right),
        \end{equation}
    where $\sigma_{\min}(A)$ denotes the smallest singular value of the matrix $A$. 
    Furthermore, suppose that the infimum in~\eqref{eq:idef01} is attained at $\hat \omega$ and let $\hat x$ be the right singular vector corresponding to the smallest singular value of $G(\hat \omega) Q^{-1}$. Then for $\hat \Delta_J=-i \hat H_{(Q\hat x, i(i\hat \omega I_n-JQ)\hat x)}$, where $\hat H$ is defined in~\eqref{def:hatH} and $\hat \Delta_R=\frac{-1}{\hat x^*QRQ\hat x}(RQ\hat x)(RQ\hat x)^*$, we have 
    \begin{equation*}
        r^{S_i}(J,R)=\sqrt{\|\hat \Delta_J\|^2+\|\hat \Delta_R\|^2}.
    \end{equation*}
    }
\end{theorem}
\begin{proof}
    By Definition~\ref{def:radii}, we have
    \begin{equation}
        r^{S_i}(J,R)=\inf \left\{ \nnrm{(\Delta_J,\Delta_R)}~:~(\Delta_J,\Delta_R)\in S_i(J,R),~ \Lambda((J+\Delta_J - (R+\Delta_R))Q) \cap i\R \neq \emptyset \right\}.
    \end{equation}
Note that for $(\Delta_J,\Delta_R)\in S_i(J,R)$, the perturbed system $(J+\Delta_J-(R+\Delta_R))Q$ remains dissipative Hamiltonian; and hence, using the spectral properties of DH systems \cite{MehMS16,MehMW18}, we obtain
\begin{align*}
        &r^{S_i}(J,R)\\
        &=\inf \left\{ \nnrm{(\Delta_J,\Delta_R)} : (\Delta_J,\Delta_R)\in S_i(J,R),~ (R+\Delta_R)Qx=0 \text{ for some eigenvector $x$ of } (J+\Delta_J)Q  \right\}\\
        &=\inf \{ \nnrm{(\Delta_J,\Delta_R)} : (\Delta_J,\Delta_R)\in S_i(J,R),~ \Delta_RQx=-RQx \text{ for some $x\in \C^n\setminus\{0\}$ satisfying }\\
        & \hspace{8cm} (J+\Delta_J)Qx=i\omega x, \omega\in \R \}\\
        &=\inf_{\omega\in\R} \inf_{x\in\C^{n}\setminus\{0\}}\Big( \inf \Big\{ \nnrm{(\Delta_J,\Delta_R)} : (\Delta_J,\Delta_R)\in S_i(J,R),~ \Delta_RQx=-RQx, (J+\Delta_J)Qx=i\omega x  \Big\}\Big)\\
        &=\inf_{\omega\in \R} \vartheta^{S_i}_{\omega},
\end{align*}
where, as in Theorem~\ref{thm:def}, in the second-to-last equation the optimization problem is divided into sublevel optimization in the variables $\omega$, $x$ and $(\Delta_J,\Delta_R)$, and for a given $\omega\in \R$, $\vartheta^{S_i}_{\omega}$ is defined as
    \begin{align}\label{eq:idefv}
        \vartheta^{S_i}_{\omega} &:= \inf_{x\in\C^{n}\setminus\{0\}}\Big( \inf \Big\{ \nnrm{(\Delta_J,\Delta_R)} : (\Delta_J,\Delta_R)\in S_i(J,R),~ \Delta_RQx=-RQx,\nonumber\\
        & \hspace{8cm} \Delta_JQx=(i\omega I_n - JQ)x  \Big\}\Big).
    \end{align}
    Considering the inner optimization in~\eqref{eq:idefv}, for given $x$ and $\omega$, using Lemma~\ref{map:herm}, there exists $\Delta_R$ and $\Delta_J$ Hermitian and skew-Hermitian, respectively, such that $\Delta_RQx=-RQx$ and $\Delta_JQx=(i\omega I_n - JQ)x$ if and only if $\Im(x^*QRQx)=0$ and $\Re(x^*Q(i\omega I_n - JQ)x)=0$, which holds  because of the structure on the matrices $J,R$ and $Q$, and the minimal norms of such $\Delta_R$ and $\Delta_J$ are given by
    \begin{equation*}
        \|\Delta_R\|=\frac{\|RQx\|}{\|Qx\|}, \quad \|\Delta_J\|=\frac{\|(i\omega I_n - JQ)x\|}{\|Qx\|}.
    \end{equation*}
    This implies that
    \begin{align}\label{norm:idef}
        \nnrm{(\Delta_J,\Delta_R)}^2&=\frac{x^*QR^2Qx}{x^*Q^2x}+\frac{x^*(i\omega I_n - JQ)^*(i\omega I_n - JQ)x}{x^*Q^2x}\\ \nonumber
        &=\frac{x^*((RQ)^*RQ + (i\omega I_n - JQ)^*(i\omega I_n - JQ))x}{x^*Q^2x}.
    \end{align}
   Using~\eqref{norm:idef} and $S_i(J,R)\subseteq \{(\Delta_J, \Delta_R):\Delta_J^*=-\Delta_J, \Delta_R^*=\Delta_R\}$ in~\eqref{eq:idefv}, we obtain
    \begin{align}\label{eq:idef1}
        (\vartheta^{S_i}_{\omega})^2&\geq \inf_{x\in \C^n\setminus\{0\}} \left\{ \frac{x^*((RQ)^*RQ + (i\omega I_n - JQ)^*(i\omega I_n - JQ))x}{x^*Q^2x}\right\} \nonumber\\ 
        &=\inf_{y\in \C^n\setminus\{0\}} \left\{ \frac{y^*Q^{-1}G(\omega)^*G(\omega) Q^{-1}y}{y^*y}\right\} \nonumber\\ 
        &= \left(\sigma_{\text{min}}\left( G(\omega) Q^{-1}\right)\right)^2,
    \end{align}
    where $y:=Qx$ and $G(\omega):=\mat{cc} RQ \\ (i\omega I_n - JQ) \rix$. {This proves the inequality in~\eqref{eq:idef01}. To prove the equality, we consider two cases.

    Case 1: when $R>0$. Let $\hat y$ be a right singular vector corresponding to the smallest singular value of $ G(\omega) Q^{-1}$ and set $\hat x = Q^{-1}\hat y$. Using Lemma~\ref{map:herm}, there exists optimal
    $(\hat \Delta_J,\hat \Delta_R)$ such that $\hat \Delta_J^*=-\hat\Delta_J$, $\hat\Delta_R^*=\hat\Delta_R$ satisfying
       \begin{align}
       \hat\Delta_JQ\hat x&=(i\omega I_n-JQ)\hat x,  \quad 
    \hat\Delta_RQ\hat x=-RQ\hat x \label{norm:ideftemp1} \\ 
        \nnrm{(\hat \Delta_J,\hat \Delta_R)}^2&=\frac{\hat x^*((RQ)^*RQ + (i\omega I_n - JQ)^*(i\omega I_n - JQ))\hat x}{\hat x^*Q^2\hat x}=\sigma_{\text{min}}\left( G(\omega) Q^{-1}\right).
        \label{norm:ideftemp2}
    \end{align}
Thus, if we show that $R+\hat \Delta_R \geq 0$, then this will imply that $(\hat \Delta_J,\hat \Delta_R) \in S_i(J,R)$, and thus in view of~\eqref{norm:ideftemp2} and \eqref{eq:idef1}, we have equality in~\eqref{eq:idef01}. 
    %
    %
    %
    %
    Note that the optimal matrix $\hat H$ attaining the minimal norm in Lemma~\ref{map:herm} has at most one negative eigenvalue, since it is either a rank one or rank two matrix, and if it has rank two, then it is easy to check that $y\pm \frac{\|y\|}{\|x\|}x$ are eigenvectors corresponding to the eigenvalues $\pm \frac{\|y\|}{\|x\|}$, respectively. This implies that $\hat\Delta_R$ has at most one negative eigenvalue. By using~\eqref{norm:ideftemp1}, we obtain
    \begin{equation*}
        (R+\hat \Delta_R)Q\hat x= RQ\hat x + \hat \Delta_RQ \hat x= RQ\hat x - RQ\hat x=0.
    \end{equation*}
    Thus, $\hat\Delta_R$ is a matrix with at most one negative eigenvalue and also satisfies that $R+\Delta_R$ is singular. This implies from~\cite[Lemma 4.4]{MehMS16} that $ R+\hat \Delta_R \geq 0$. 

   Case 2: when $R\geq 0$ and singular. In this case, for any $\epsilon > 0$, $R_\epsilon=R+\epsilon I$ is positive definite and by following the arguments of Case 1, we have 
   \[
   r^{S_i}(J,R_\epsilon)=\inf_{\omega\in \R} \vartheta^{S_i}_{\omega,\epsilon}= \inf_{\omega\in \R}\sigma_{\text{min}}\left( G(\omega,\epsilon) Q^{-1}\right),
   \]
   where $\vartheta^{S_i}_{\omega,\epsilon}$ is defined by~\eqref{eq:idefv} when $R$ is replaced by $R_\epsilon$, and $G(\omega,\epsilon)Q^{-1}:=\mat{cc} R+\epsilon I  \\ (i\omega Q^{-1} -J) \rix$. Since singular values depend continuously on matrix entries, the function $\sigma_{\text{min}}\left( G(\omega,\epsilon) Q^{-1}\right)$
   is continuous in $\epsilon$ and $\omega$. Also for any
   $\epsilon >0$, the infimum of $\sigma_{\text{min}}\left( G(\omega,\epsilon) Q^{-1}\right)$ over $\omega$ is attained in a compact interval. This implies that 
   $r^{S_i}(J,R_\epsilon)=\inf_{\omega\in \R}\sigma_{\text{min}}\left( G(\omega,\epsilon) Q^{-1}\right)$ is continuous with respect to $\epsilon$. Thus, we have
   \begin{eqnarray}
   r^{S_i}(J,R)=\lim_{\epsilon \to 0} r^{S_i}(J,R_\epsilon)=
   \lim_{\epsilon \to 0} \inf_{\omega\in \R}\sigma_{\text{min}}\left( G(\omega,\epsilon) Q^{-1}\right) = \inf_{\omega\in \R}\sigma_{\text{min}}\left( G(\omega) Q^{-1}\right). 
\end{eqnarray}
This completes the proof. 
    }
\end{proof}

\begin{remark}
    As we have already noted earlier, the structured stability radius $r^{S_i}(J,R)$ gives the lower bound 
    \begin{equation}
        r^{S_i}(J,R) \leq r^{S_d}(J,R).
    \end{equation}
    A natural question that arises is how good this lower bound is as compared to the one obtained in Corollary~\ref{cor:def}. Using the variational inequality for eigenvalues of Hermitian matrices, see \cite{GolV96}, we see that
    \begin{align*}
        r^{S_i}(J,R) &\geq \inf_{\omega\in \R} \lambda_{\min}(Q^{-1}G^*GQ^{-1})\\
        & = \inf_{\omega\in \R} \lambda_{\min}( R^2+ Q^{-1}(i\omega I_n - JQ)^*(i\omega I_n - JQ)Q^{-1})\\
        &\geq \lambda_{\min}(R)^2 + \inf_{\omega\in \R} \lambda_{\min}(Q^{-1}(i\omega I_n - JQ)^*(i\omega I_n - JQ)Q^{-1}).
    \end{align*}
    {This shows that the bound for $r^{S_d}(J,R)$ in Corollary~\ref{cor:def} is weaker than the bound $\sigma_{\min}(G(\omega)Q^{-1})\leq r^{S_i}(J,R) \leq r^{S_d}(J,R)$.}
\end{remark}

\subsection{The structured stability radius \texorpdfstring{$r^{S}(J,R)$}{TEXT}}\label{sec:Herm}
In this subsection, we discuss the stability radius $r^{S}(J,R)$ defined in~\eqref{eq:radS}, while considering perturbations from the set $S$  in~\eqref{eq:setS}, allowing  $R+\Delta_R$ to become indefinite. Note that in this case the DH structure may be destroyed. Thus, to obtain a characterization of $r^{S}(J,R)$ we employ the eigenvalue backward error $\eta^{S}(J,R,i\omega)$  defined in~\eqref{eq:backerror} to determine
\begin{equation}\label{eq:hermerror}
    r^{S}(J,R)=\inf_{\omega\in \R} \eta^{S}(J,R,i\omega),
\end{equation} 
where we first derive a computable formula for $ \eta^{S}(J,R,i\omega)$, and then using it in~\eqref{eq:hermerror}, we obtain a characterization for $r^S(J,R)$. 

{
\begin{lemma}\label{lem:det}
    Consider a DH system of the form~\eqref{DHsystem}, and let $\lambda\in \C$ be such that $M_\lambda:=((J-R)Q-\lambda I_n)^{-1}$ exists. Further, let $(\Delta_J,\Delta_R)\in S(J,R)$ defined in~\eqref{eq:setS}. Then the following statements are equivalent:
    \begin{enumerate}
        \item[(i)] \text{det}$((J+\Delta_J-(R+\Delta_R))Q-\lambda I_n)=0$.
        \item[(ii)] There exists vectors $v_J,v_R$ satisfying $v_J-v_R\neq 0$ such that $\Delta_JQM(v_J-v_R)=v_J$ and $\Delta_RQM(v_J-v_R)=v_R$.
    \end{enumerate}
\end{lemma}
\begin{proof}
    $(i)\implies (ii)$. First suppose that $(i)$ holds, then the determinant condition of $(i)$ implies that there exists $x\neq 0$ such that $((J+\Delta_J-(R+\Delta_R))Q-\lambda I_n)x=0$. Define $v_J:=\Delta_JQx$ and $v_R:=\Delta_RQx$, then
    \[
    0=((J+\Delta_J-(R+\Delta_R))Q-\lambda I_n)x=((J-R)Q-\lambda I_n)x + v_J - v_R .
    \]
    Clearly $v_J-v_R\neq 0$, as $((J-R)Q-\lambda I_n)$ is invertible. On pre-multiplying the last equation with $\Delta_JQM$ and $\Delta_RQM$, we obtain $v_J=\Delta_JQM(v_J-v_R)$ and $v_R=\Delta_RQM(v_J-v_R)$.\\
    $(ii)\implies (i)$. Suppose that $(ii)$ holds. Then
    \begin{align*}
    ((J+\Delta_J-(R+\Delta_R))Q-&\lambda I_n)M(v_J-v_R)\\
    &=((J-R)Q-\lambda I_n)M(v_J-v_R) + (\Delta_J-\Delta_R)QM(v_J-v_R) \\
    &= -(v_J-v_R) + (v_J-v_R) \\
    &=0.
    \end{align*}
    This implies that \text{det}$((J+\Delta_J-(R+\Delta_R))Q-\lambda I_n)=0$, since $M(v_J-v_R)\neq 0$.
\end{proof}}
%
%
%

\begin{theorem}\label{thm:Herm}
    Consider a DH system of the form~\eqref{DHsystem}, and let $\lambda\in \C$ be such that {$M_\lambda:=-((J-R)Q-\lambda I_n)^{-1}$ exists.} Further, define
    \begin{equation}\label{mat:H}
        H^{(\lambda)}:=[I_n~-I_n]^*M_\lambda^*Q^2M_\lambda[I_n~-I_n],
    \end{equation}
    \begin{equation}\label{mat:H0H1}
        H_0^{(\lambda)}:=\mat{cc} QM_\lambda + M_\lambda^*Q & -QM_\lambda\\ -M_\lambda^*Q & 0 \rix, \quad H_1^{(\lambda)}:=i\mat{cc} 0 & -M_\lambda^*Q \\ QM_\lambda & -QM_\lambda+M_\lambda^*Q \rix.
    \end{equation}
    Then
    \begin{equation}\label{eq:Herm}
        \eta^{{S}}(J,R,\lambda)=\left( \min_{t_0,t_1\in \R} \lambda_{\max}(H^{(\lambda)}+t_0H_0^{(\lambda)}+t_1H_1^{(\lambda)}) \right)^{-\frac{1}{2}}.
    \end{equation}
\end{theorem}
\begin{proof}
    The dependence on $\lambda$ in the matrices $M_\lambda$, $H^{(\lambda)}, H_0^{(\lambda)}$ and $H_1^{(\lambda)}$ has been highlighted for  future reference only; in the proof we omit this dependency and we will use the abbreviations $M$ for $M_\lambda$, $H$ for $H^{(\lambda)}$, $H_0$ for $H_0^{(\lambda)}$ and $H_1$ for $H_1^{(\lambda)}$. By Definition~\ref{def:radii}, we have
    \begin{align}\label{eq:Herm1}
        \eta^{S}(J,R,\lambda) &= \inf\{ \nnrm{(\Delta_J,\Delta_R)} : (\Delta_J,\Delta_R)\in S(J,R), \lambda \in \Lambda((J+\Delta_J-(R+\Delta_R))Q) \}\nonumber\\
        &=\inf\{ \nnrm{(\Delta_J,\Delta_R)} : (\Delta_J,\Delta_R)\in S(J,R), \text{det}((J+\Delta_J-(R+\Delta_R))Q-\lambda I_n)=0 \}.
    \end{align}
    {Using Lemma~\ref{lem:det} in~\eqref{eq:Herm1}, we obain
    %
    \begin{align}\label{eq:Herm2}
        \eta^{S}(J,R,\lambda)&= \inf_{v_J,v_R\in\C^{n}} \inf\{ \nnrm{(\Delta_J,\Delta_R)} : (\Delta_J,\Delta_R)\in S(J,R), v_J-v_R\neq 0,\nonumber\\ & \hspace{4cm} \Delta_JQM(v_J-v_R)=v_J, \Delta_RQM(v_J-v_R)=v_R \}.
    \end{align}
    }
  Applying Lemma~\ref{map:herm} for the minimal norm skew-Hermitian and Hermitian mappings sending one vector to another, the inner optimization problem in~\eqref{eq:Herm2} can be rewritten as
    \begin{align}\label{eq:Herm3}
        (\eta^{S}(J,R,\lambda))^2&= \inf_{v_J,v_R\in\C^{n}}\bigg\{ \frac{\|v_J\|^2}{\|QM(v_J-v_R)\|^2} + \frac{\|v_R\|^2}{\|QM(v_J-v_R)\|^2} : v_J-v_R\neq 0,\nonumber\\ & \hspace{4cm} \Re{(v_J^*QM(v_J-v_R))}=0, \Im{(v_R^*QM(v_J-v_R))}=0 \bigg\}.
    \end{align}
    Setting $v:=[v_J^T~v_R^T]^T$, the condition $\Re{(v_J^*QM(v_J-v_R))}=0$ in~\eqref{eq:Herm3} can be equivalently written as $v_J^*QM(v_J-v_R)+(v_J^*QM(v_J-v_R))^*=0$, which  can be expressed as $v^*H_0v=0$, where $H_0$ is defined in~\eqref{mat:H0H1}. Similarly, the condition $\Im{(v_R^*QM(v_J-v_R))}=0$ can be expressed as $v^*H_1v=0$, where $H_1$ is defined in~\eqref{mat:H0H1}. Also, the objective function in~\eqref{eq:Herm3} can be rewritten as 
    \begin{equation}
        \frac{\|v_J\|^2}{\|QM(v_J-v_R)\|^2} + \frac{\|v_R\|^2}{\|QM(v_J-v_R)\|^2}=\frac{\|v\|^2}{\|QM[I_n~-I_n]v\|^2}=\frac{v^*v}{v^*Hv},
    \end{equation}
    where $H$ is as defined in~\eqref{mat:H}. Also, observe that $v_J-v_R \neq 0$ if and only if $\|QM(v_J-v_R)\|\neq 0$ if and only if $v^*Hv\neq 0$. Using these expressions in~\eqref{eq:Herm3}, we obtain
    \begin{align}\label{eq:Herm4}
        (\eta^{S}(J,R,\lambda))^2&=\inf \left\{ \frac{v^*v}{v^*Hv} : v\in\C^{2n}, v^*Hv\neq 0, v^*H_0v=0,v^*H_1v=0  \right\}\nonumber \\ 
        &=\left( \sup \left\{ \frac{v^*Hv}{v^*v} : v\in\C^{2n}\setminus\{0\}, v^*H_0v=0,v^*H_1v=0 \right\} \right)^{-1},
    \end{align}
    where in the last equation we removed the constraint $v^*Hv\neq 0$, because the eigenvalue backward error $\eta^{S}(J,R,\lambda)$ is finite (as $\eta^S(J,R,\lambda)\leq \nnrm{(J,R)}$); hence, it will not be attained by vectors $v$ satisfying $v^*Hv=0$, and therefore the condition $v^*Hv\neq 0$ is superfluous. Next, we will apply {Theorem 2.6 \cite{BorKMS14}} in~\eqref{eq:Herm4}. For this we have to check that  $t_0H_0+t_1H_1$ is indefinite for every $(t_0,t_1)\in \R^{2}\setminus \{0\}$. Suppose, on the contrary that there exists $(\hat{t}_0,\hat{t}_1)\in \R^2$ such that $\hat{t}_0H_0+\hat{t}_1H_1$ is semidefinite. 
  Then, with $N:=\mat{cc} I_n & I_n \\ 0 & I_n \rix$, we obtain
    \begin{equation}
        N^*(\hat{t}_0H_0+\hat{t}_1H_1)N= \mat{cc} \hat{t}_0QM + t_0 (QM)^* & ((\hat{t}_0+i\hat{t}_1)QM)^* \\ (\hat{t}_0+i\hat{t}_1)QM & 0 \rix.
    \end{equation}
    Since $\hat{t}_0H_0+\hat{t}_1H_1$ is semidefinite, this implies that $(\hat{t}_0+i\hat{t}_1)QM = 0$, which further implies that $\hat{t}_0+i\hat{t}_1 = 0$, since $M$ and $Q$ are invertible. This gives $\hat t_0=0$ and $\hat t_1=0$, and proves that $t_0H_0 + t_1H_1$ is indefinite for every $(t_0,t_1)\in\R^2\setminus\{0\}$ and hence the result follows using {Theorem 2.6 \cite{BorKMS14}}.
\end{proof}

We then have the following corollary.
\begin{corollary}
    Consider an asymptotically stable DH system of the form~\eqref{DHsystem}. Then
    \begin{equation}\label{eq:Herm01}
        r^{S}(J,R)=\inf_{\omega\in\R} \left( \min_{t_0,t_1\in \R} \lambda_{\max}(H^{(i\omega)}+t_0H_0^{(i\omega)}+t_1H_1^{(i\omega)}) \right)^{-\frac{1}{2}},
    \end{equation}
    where $H^{(i\omega)}, H_0^{(i\omega)}, \text{ and } H_1^{(i\omega)}$ are defined in~\eqref{mat:H} and~\eqref{mat:H0H1}. 
    {
    Furthermore, suppose that the infimum in~\eqref{eq:Herm01} is attained at $\hat \omega, \hat t_0, \hat t_1$ and let $\hat x$ be the eigenvector corresponding to the largest eigenvalue of $(H^{(i\hat \omega)}+\hat t_0H_0^{(i\hat \omega)}+\hat t_1H_1^{(i\hat \omega)})$. Then for $\hat \Delta_J=-i \hat H_{(Q\hat x, i(i\hat \omega I_n-JQ)\hat x)}$ and $\Delta_R= \hat H_{(Q\hat x, -RQ\hat x)}$, where $\hat H$ is defined in~\eqref{def:hatH}, we have
    \begin{equation*}
        r^{S}(J,R)=\sqrt{\|\hat \Delta_J\|^2+\|\hat \Delta_R\|^2}.
    \end{equation*}}
\end{corollary}

\begin{remark}\label{rem:optimal}
Let $A\in \mathbb C^{n,n}$ have all its eigenvalues in the open left half complex plane. Then for a positive definite solution $X=X^*> 0$ of the strict Lyapunov inequality
\begin{equation}\label{lyaineqcc}
- A^* X-X A> 0,
\end{equation}
let $T$ be the positive definite square root $X^{\frac 12}$ or the Cholesky factor of $X$. Multiplying $\dot x =Ax$ from the left by $T$ and 
performing  a  change of basis $y=Tx$ {we obtain a DH system $\dot y=(J-R)y$ with $$J=\frac{1}{2}(TAT^{-1}-T^{-1}A^*T),\; R=-\frac{1}{2}(TAT^{-1}+T^{-1}A^*T), \text{ and } Q=I,$$
where the right hand side has a positive Hermitian part $R>0$.}

Since the solution of \eqref{lyaineqcc} is not unique, one may ask the question whether there is a solution $X$ that leads, in some sense, to an 'optimal' DH representation. It has been suggested  in \cite{AchAC23}  to solve the Lyapunov-like equation
\begin{equation}\label{lyaeq}
- A^* X-X A=-2\mu X,
\end{equation}
where $\mu<0$ is the spectral abscissa, i.e. the maximal real part of an eigenvalue of $A$. This choice of the representation then leads to a (maximal) smallest eigenvalue of $R$ that is equal to $\mu$, which means that the field of values of $J-R$ lies completely on the left of the line $\mu+ i \mathbb R$. Furthermore, it then follows that that the decay of the spectral norm of the fundamental solution matrix $\|e^{(J-R)t}\|$ is then bounded by $1+\mu t+ {\mathcal O}(t^2)$, see also \cite{AchAM21,AchAM23}.

\end{remark}

\section{Minimization of functions involving generalized Rayleigh quotients}\label{sec:minSRQ2}
In this section we consider the inner optimization problem in $r^{S_d}(J,R)$ in Theorem~\ref{thm:def}, i.e., the minimization problem involving the generalized Rayleigh quotient of the form
\begin{equation}\label{minSRQ2}
    \min_{x\in \C^n, \|x\|=1} f(x) := \left(\frac{x^*R^2x}{x^*Rx}\right)^2 + x^*Px,
\end{equation}
where $R$ is an Hermitian positive semidefinite matrix and $P:=Q^{-1}(i\omega I_n - JQ)^*(i\omega I_n - JQ)Q^{-1}$, with $\omega\in \R$, $J$ skew-Hermitian and $Q$ Hermitian positive definite. Problem~\eqref{minSRQ2}, can be solved by conventional Riemannian optimization techniques, such as the Riemannian trust region method; see, e.g.,~\cite{AbsMS09}. However, these general-purpose Riemannian optimization methods {often converge to local minimizers}, and they do not fully exploit the special form of the objective function for analysis and computation. In~\cite{LuPSB25}, a Joint Numerical Range (JNR)-based approach was suggested for the solution of a  minimization problem similar to~\eqref{minSRQ2} with objective function
\begin{equation*}
    \frac{x^*Ax}{x^*(\alpha_1I_n + \beta_1C)x} + \frac{x^*Bx}{x^*(\alpha_2I_n + \beta_2C)x},
\end{equation*}
where $\alpha_i, \beta_i\in \R$ for $i=1,2$ and $A,B,C$ are Hermitian positive semidefinite matrices. It was shown that via the JNR approach, the reduced problem has fewer local minimizers than the original one. Following the same approach, we reformulate the minimization problem~\eqref{minSRQ2} into a new optimization problem over a JNR of the matrices $R^2,R$ and $P$. The convexity in the JNR also allows for the development of a nonlinear eigenvector approach to efficiently solve the optimization problem.

Let $\mathcal{K} := (R^2,R,P)$ and define the JNR associated with $\mathcal K$ as 
\begin{equation}\label{JNR}
    W(\mathcal K) := \left\{ \rho_{\mathcal K}(x) : x\in \C^n, \|x\|=1 \right\},
\end{equation}
where $\rho_{\mathcal K} : \C^n \to \R^3$ consists of quadratic forms of the Hermitian matrices in $\mathcal K$
\begin{equation}
    \rho_{\mathcal K}(x) := [x^*R^2x,x^*Rx,x^*Px]^T.
\end{equation}
It is well-known that  $W(\mathcal K)$ is a closed and connected region in $\R^3$, and it is a convex set if the size $n$ of the matrices $R,P$ satisfies $n\geq 3$, see~\cite{AuT83,MuT20}.\\

The objective function $f(x)$ in~\eqref{minSRQ2} then can be written as 
\begin{equation}
    f(x) = h(\rho_{\mathcal K} (x)),
\end{equation}
where $h: \R^3 \to \R $ is given by 
\begin{equation}\label{func:h}
    h(z):= \left(\frac{z_1}{z_2} \right)^2 + z_3.
\end{equation}
Therefore, by using the intermediate variable $z=\rho_{\mathcal K}(x)$, the minimization problem~\eqref{minSRQ2} can be reformulated as 
\begin{equation}\label{JNRproblem}
    \min_{z\in W(\mathcal K)} h(z),
\end{equation}
where $W(\mathcal K)$ is the JNR associated with $\mathcal K$ defined in~\eqref{JNR}. 

In the JNR minimization~\eqref{JNRproblem}, the variable $z\in \R^3$ is real, in contrast to a complex and n-dimensional variable vector $x\in \C^n$ of the original minimization problem~\eqref{minSRQ2}. Typically $n\gg3$ in practice, so the new feasible region $W(\mathcal{K})$ is a convex set. This facilitates analysis, computation, and visualization of the JNR minimization~\cite{LuPSB25}.

The following lemma from~\cite{LuPSB25} which holds for the minimization problem~\eqref{minSRQ2}, indicates that the JNR minimization~\eqref{JNRproblem} is superior to the
original minimization problem~\eqref{minSRQ2}.
\begin{lemma}\label{lem:NEPvJNR}
    Consider the optimization problems~\eqref{minSRQ2} and~\eqref{JNRproblem}.
    \begin{enumerate}
        \item A vector $x_\star\in \C^n, \|x\|=1$ is a global minimizer of n~\eqref{minSRQ2} if and only if $z_\star = \rho_{\mathcal K}(x_\star)$ is a global minimizer of the JNR minimization problem~\eqref{JNRproblem}.
        \item If $z_\star\in \R^3$ is a local minimizer of the JNR minimization problem~\eqref{JNRproblem}, then any $x_\star$ with $\rho_{\mathcal K}(x_\star) = z_\star$ is a local minimizer of~\eqref{minSRQ2}.
    \end{enumerate}
\end{lemma}
We will now establish a variational characterization for the local minimizers of ~\eqref{JNRproblem} by exploiting the convexity of the JNR. This characterization can  be equivalently expressed as a nonlinear eigenvalue problem with eigenvector dependency (NEPv). The following theorem is an adaption of~\cite[Theorem 3.2]{LuPSB25} for our problem.
\begin{theorem}\label{thm:NEPvchar}
    Let $M=(R^2,R,P)$, $W(\mathcal K)$ be defined by~\eqref{JNR}, and $h$ by~\eqref{func:h}. Suppose that $z_{\star}$ is a local minimizer of $h$ in $W(\mathcal K)$, and that $h$ is differentiable at $z_{\star}$. Then
    \begin{equation}\label{eq:varchar}
        \min_{z\in W(\mathcal K)} \nabla h(z_{\star})^Tz = \nabla h(z_{\star})^Tz_{\star}.
    \end{equation}
    Moreover,~\eqref{eq:varchar} holds if and only if $z_{\star} = \rho_{\mathcal K}(x_{\star})$ for an $x_{\star}\in \C^n, \|x_{\star}\|=1$ satisfying the NEPv:
    \begin{equation}\label{eq:NEPv}
        H(x)x = \mu x
    \end{equation}
    where $H(x) \in \C^{n,n}$ is the Hermitian matrix given by
    \begin{equation}\label{eq:NEPvmat}
        H(x) := 2\frac{x^*R^2x}{(x^*Rx)^2} R^2 - 2 \frac{(x^*R^2x)^2}{(x^*Rx)^3} R + P
    \end{equation}
    and $\mu$ is the smallest eigenvalue of $H(x)$.
\end{theorem}
\begin{proof}
    The proof follows by taking $h(z) = \left(\frac{z_1}{z_2}\right)^2 + z_3$ in~\cite[Theorem 3.2]{LuPSB25}.
\end{proof}

In the literature, nonlinear eigenvalue problem characterizations have been explored in various optimization problems with orthogonality constraints; see~\cite{BaL24} and the references therein. These characterizations allow for efficient solutions of the optimization problem by exploiting state-of-the-art eigensolvers. Particularly in~\cite{Zh13,Zh14} the authors propose such characterizations that apply to the optimization of the sum of Rayleigh quotients
\begin{equation}\label{srq21}
    \min_{x\in \R^n, \|x\|=1} \frac{x^TH_1x}{x^TH_2x} + x^TH_3x,
\end{equation}
where $H_i\in \R^{n,n}$ for $i=1,2,3$ are Hermitian positive definite matrices, and in~\cite{LuPSB25} such characterizations are suggested for the optimization of sums of generalized Rayleigh quotients 
\begin{equation}\label{srq22}
    \min_{x\in \C^n, \|x\|=1} \frac{x^*A_1x}{x^*(\alpha_1 I_n + \beta_1 A_2)x} + \frac{x^*A_3x}{x^*(\alpha_2 I_n + \beta_2 A_2)x}
\end{equation}
where $A_i\in \C^{n,n}$ are Hermitian positive semidefinite. Notice that the minimization problem~\eqref{minSRQ2} cannot be generally reduced to the form~\eqref{srq21} or~\eqref{srq22}, due to the square term in~\eqref{minSRQ2}. Therefore, the previous analysis does not directly apply to the minimization~\eqref{minSRQ2}.

Note that in Theorem~\ref{thm:NEPvchar}, the function $h(z)$ needs to be differentiable at the local minimizer $z_\star$. Non-differentiability may occur if the fraction  in $h(z_\star)$ becomes $\frac{0}{0}$. The following theorem considers this case and shows that such local minimizers can be obtained easily.

\begin{theorem}\label{thm:NEPvnondiff}
    Let $z_\star$ be a local minimizer of $h$ over $W(\mathcal K)$.
    \begin{enumerate}
        \item If $R>0$, then $z_\star$ must be a differentiable point for $h(z)$ and hence it admits the NEPv characterization given by Theorem~\ref{thm:NEPvchar}.
        \item If $R\geq 0$ is singular, then $z_\star$ may be a non-differentiable point for $h(z)$ and in that case it must be expressed by $z_\star = \rho_{\mathcal K}(Uv)$, where $U$ is an orthonormal basis matrix of the nullspace of $R$, and $v$ is the eigenvector for the smallest eigenvalue of the matrix $U^*PU$.
    \end{enumerate}
\end{theorem}
\begin{proof}
    The function $h$ is differentiable at $z$ if and only if $z_2=x^*Rx\neq 0$. If $R>0$, then by parameterizing $z_\star=\rho_{\mathcal K}(x_\star)$, we must have $x_\star^*Rx_\star \neq 0$ and hence $h$ is differentiable at $z_\star$.

    Suppose now that $h$ is not differentiable at $z_\star = \rho_{\mathcal K}(x_\star)$, i.e., the denominator term in~\eqref{func:h} is $(z_2)^2=(x_\star^*Rx_\star)^2=0$, so we have $x_\star\in \text{null}(R) = \text{range}(U)$. As $R^2$ and $R$ have the same nullspaces, by the local optimality of $z_\star$ and Lemma~\ref{lem:NEPvJNR}, $x_\star$ must be a local minimizer of $f(x)$ restricted to the subspace range$(U)$
    \begin{equation}\label{eq123}
        \min_{\substack{x=Uv\\ \|v\|=1}} f(x) = \min_{\substack{x=Uv\\ \|v\|=1}} x^*Px = \min_{\substack{\|v\|=1}} v^*(U^*PU)v,
    \end{equation}
    where $P$ is defined in~\eqref{minSRQ2}. The first equality in~\eqref{eq123} is due to $\frac{x^*R^2x}{x^*Rx} = \frac{0}{0} = 0$ for all $x\in \text{range}(U)$. Hence, the local minimizer of~\eqref{JNRproblem} is $z_\star = \rho_{\mathcal K}(Uv)$, where $v$ is the eigenvector corresponding to the smallest eigenvalue of $U^*PU$.
\end{proof}

To find the minimal solution of~\eqref{minSRQ2}, we can solve the NEPv~\eqref{eq:NEPv}, provided that it has a solution, along with the smallest eigenvalue and the corresponding eigenvector of the matrix $U^*PU$. We can then select  $x_\star$ with the minimal objective value $f(x_\star)$ as the solution. 
Thus, in view of Theorems~\ref{thm:NEPvnondiff} and~\ref{thm:def}, we have the following result that gives an explicit formula for the structured stability radius $r^{S_d}(J,R)$ in terms of solution of an NEPv.
\begin{theorem}\label{cor:def1}
    Consider a DH system of the form~\eqref{DHsystem} and consider the NEPv $H(x)x=\mu x$ defined in~\eqref{eq:NEPv} for the matrix $H(x)$ defined in~\eqref{eq:NEPvmat}.
    \begin{enumerate}
        \item If $R>0$, then
        \begin{equation*}
            r^{S_d}(J,R) = \left( \frac{x_{\star}^*R^2x_{\star}}{x_{\star}^*Rx_{\star}}\right)^2 + \inf_{\omega \in \R} x_{\star}^*Q^{-1}(i\omega I_n - JQ)^*(i\omega I_n - JQ)Q^{-1}x_{\star},
        \end{equation*}
        {where $x_{\star}$ is an eigenvector corresponding to the smallest eigenvalue of the NEPv $H(x)x=\mu x$.}
        \item If $R\geq 0$ is singular and if a solution $x_{\star}$ to the NEPv $H(x)x=\mu x$, where $\mu$ is the smallest eigenvalue, exists, 
        then
        \begin{align*}
            r^{S_d}(J,R) &= \min \bigg\{ \left( \frac{x_{\star}^*R^2x_{\star}}{x_{\star}^*Rx_{\star}}\right)^2 + \inf_{\omega\in \R} x_{\star}^*Q^{-1}(i\omega I_n - JQ)^*(i\omega I_n - JQ)Q^{-1}x_{\star},\\ &\hspace{4cm} \inf_{\omega\in \R} \lambda_{\min}(U^*Q^{-1}(i\omega I_n - JQ)^*(i\omega I_n - JQ)Q^{-1}U) \bigg\}.
        \end{align*}
        \item If $R\geq 0$ is singular and the NEPv $H(x)x=\mu x$ is not solvable, 
        then
        \begin{align*}
            r^{S_d}(J,R) &= \inf_{\omega\in \R} \lambda_{\min}(U^*Q^{-1}(i\omega I_n - JQ)^*(i\omega I_n - JQ)Q^{-1}U),
        \end{align*}
        where 
        $U$ is the orthonormal basis matrix of the nullspace of matrix $R$.
    \end{enumerate}
\end{theorem}

\section{Numerical experiments}\label{sec:numerical}
In this section, we present some numerical experiments to illustrate our results on structured stability radii. These numerical experiments emphasize that the stability radii under structure-preserving perturbations can be much larger than those under general perturbations. 
\begin{itemize}
    \item To compute the stability radius $r^{S_d}(J,R)$ obtained in Theorem~\ref{thm:def}, we use the NEPv characterization for the inner optimization discussed in Section~\ref{sec:minSRQ2}, and for the outer optimization, we use the function \texttt{fminsearch} in MATLAB. We use the level-shifted Self-Consistent-Field (SCF) iteration~\cite{BaL24, YanMW07, Zh14} to solve the NEPv~\eqref{eq:NEPv}: Starting from an initial guess $x_0\in\mathbb R^n$, we iteratively solve Hermitian eigenvalue problems
    \begin{equation}\label{eq:lsscf}
    [H(x_k) + \sigma_k x_kx_k^* ] x_{k+1} = \lambda_{k+1} x_{k+1},
    \end{equation}
    for $k=0,1,\dots$, where $\sigma_k\in \mathbb R$ is a given level-shift, and $x_{k+1}$ is the eigenvector corresponding to the smallest eigenvalue $\lambda_{k+1}$ of the Hermitian matrix $H(x_k) + \sigma_k x_kx_k^*$. 
    The equation~\eqref{eq:lsscf} reduces to the plain SCF iteration if $\sigma_k\equiv 0$. Using a nonzero level-shift $\sigma_k$ often helps to keep the plain SCF stable and to speed up the convergence process. In particular, a plain SCF may not  converge; however, for sufficiently large $\sigma_k$, a level-shifted SCF is always locally convergent under mild assumptions; see, e.g.,~\cite{BaL24}. In our implementation, we adaptively select level-shifts $\sigma_k$ by trying sequentially
    \begin{equation}\label{eq:lss}
    \sigma_k = 0,\, 2\delta_k,\, 2^2\delta_k,\, 2^3\delta_k,\dots,
    \end{equation}
    until $ h(x_{k+1}) < h(x_k)$ is achieved, where $\delta_k = \lambda_{2}(H(x_k)) - \lambda_{1}(H(x_k))$ is the eigenvalue gap between the smallest and second smallest eigenvalue of $H(x_k)$. This use of $\sigma_k=2\delta_k$ is common practice; see, e.g.,~\cite{YanMW07,Zh14}. Often, a plain SCF step with $\sigma_k=0$ can produce a reduced $h(x_{k+1})$, and then there is no need to actually enter the selection loop~\eqref{eq:lss} for $\sigma_k$. Finally, the level-shifted SCF~\eqref{eq:lss} is considered to have converged if the relative residual norm satisfies
    \begin{equation}\label{eq:reltol}
    \frac{\|H(x_k) x_k - s_k x_k\|}{\|H(x_k)\|_1 + 1}\leq 10^{-10},
    \end{equation}
    where $s_k = x_k^*H(x_k) x_k$ and $\|\cdot\|_1$ denote the matrix 1-norm (i.e., maximal absolute column sum).

    \item {
    We compute the stability radius $r^{S_i}(J, R)$ obtained in Theorem~\ref{thm:idef} using the level-set technique~\cite{BoyBK89}. At each iteration, a bisection scheme updates the candidate level $s$, while the corresponding level set $\{\omega~:~\sigma_{\min}(M(\omega)Q^{-1})\leq s\}$ is computed by solving the structured quadratic eigenvalue problem
    \[
    ((M(\omega)Q^{-1})^*M(\omega)Q^{-1}-s^2I)x=0
    \]
    for the real frequencies $\omega$. The process is repeated until the desired tolerance is achieved. The unstructured stability radius $r(J, R)$ obtained in Theorem~\ref{thm:unstr} is computed analogously using the same level-set framework, with the level sets defined by the largest singular value instead of the smallest singular value.
    %
    .}
    
    \item For the stability radius $r^S(J,R)$ obtained in Theorem~\ref{thm:Herm}, we used the convex programming package CVX in MATLAB for the inner optimization and \emph{fminsearch} for the outer optimization.
    
\end{itemize}

{
\begin{remark}\label{rem:place1}{
The computational approach that is used to evaluate the stability radius $r^{S_d}(J, R)$ in Theorem~\ref{thm:def} does not guarantee global optimality. More precisely, for a fixed value of $\omega$, the NEPv formulation is solved using an SCF-type iteration, which generally converges to a stationary point and hence may only yield a local minimizer of the objective function in~\eqref{eq:NEPv}. Furthermore, the outer minimization over $\omega \in \mathbb{R}$ introduces an additional layer of nonconvexity, and the use of the MATLAB function \texttt{fminsearch} does not guarantee convergence to the exact global optimum. We therefore note that the computed value in the following examples should be regarded as a heuristic estimate of $r^{S_d}(J, R)$. }
\end{remark}

\begin{remark}
    For a fixed value of $\omega$, the second term in the objective function of~\eqref{eq:def} can be expressed as ${(\sigma_{\min}(i\omega Q^{-1}-J))}^2$. Thus, if the first Rayleigh quotient term is absent, the inner optimization problem in~\eqref{eq:def} reduces to the computation of the smallest singular value of a parameter-dependent matrix, for which globally convergent level-set techniques are more suitable. However, this situation corresponds exactly to the structured perturbation setting in which perturbations are allowed only in one of the matrices $R$ or $J$. This particular case was studied in~\cite{MehMS16}.
The development of globally reliable level-set type algorithms for the general problem, where both the
Rayleigh quotient terms are present, remains an important direction for future work.
\end{remark}

The codes and data for the examples presented below are available at \url{https://gitlab.mpi-magdeburg.mpg.de/prajapati/stability-radii.git}.

\begin{example}\label{exam:rev1}
  To assess the practical quality of the proposed approximation, we consider a toy example of a DH system in staircase form that is stable but not asymptotically stable, possessing eigenvalues on the imaginary axis. 
    \begin{equation*}
    	\tilde J=\mat{ccc} J_{11} & J_{12} &\\ -J_{12}^* & J_{22} & \\ & & J_{33} \rix, \quad \tilde R=\mat{ccc} R_{11} & & \\ & 0 & \\ & & 0 \rix,
    \end{equation*}
    where $J_{11}\in \C^{2,2}$ is skew-Hermitian, $J_{12}\in \C^{2,1}$, $J_{22},J_{33}\in i\R$, and $R_{11}\in \C^{2,2}$ is Hermitian positive definite.
    We apply structured perturbations $\tilde \Delta_J, \tilde \Delta_R$ of the form
    \begin{equation*}
    	\tilde \Delta_J = \mat{ccc} 0_{2} &  &\\  &  & \epsilon_1 \\ & \epsilon_1 & \rix, \quad \tilde \Delta_R=\mat{ccc} 0_{2} & & \\ & \epsilon_2 & \\ & & \epsilon_2 \rix,
    \end{equation*}
    where $\epsilon_1\in \C$ and $\epsilon_2>0$. These perturbations shift the imaginary eigenvalues to the left half of the complex plane, making the system asymptotically stable. The structured stability radii are then computed for the asymptotically stable system
    \[
    A=P^*(\tilde J + \tilde \Delta_J - R - \tilde \Delta_R)P =: (J-R),
    \]
    where $P \in \C^{4,4}$ is a unitary matrix.
    We consider three scenarios: (i) $\epsilon_1,\epsilon_2\neq 0$, (ii) $\epsilon_1\neq 0, \epsilon_2=0$, (iii) $\epsilon_1=0, \epsilon_2\neq 0$.
    Since the perturbations that are used to stabilize the system are known explicitly, the norm of the applied perturbation provides an a priori upper bound for the corresponding structured stability radius. The results, reported in Table~\ref{tab:placeholder}, show that for all tested examples, the computed values of $r^{S_d}(J, R)$ remain consistently below these upper bounds. Although this does not constitute proof of global optimality, it indicates that the proposed computational approach produces meaningful and practically useful estimates of the structured stability radius. 
    \begin{table}[]
        \centering
        \begin{tabular}{|c|c|c|c|c|c|}
        \hline
            $\epsilon_1,\epsilon_2$ & $r(J,R)$ & $r^{S}(J,R)$ & $r^{S_i}(J,R)$ & $r^{S_d}(J,R)$ & $\sqrt{\|\tilde \Delta_J\|^2 + \|\tilde \Delta_R\|^2}$ \\ \hline
            $\epsilon_1=1+1i, \epsilon_2=0.5$ & 0.2258 & 0.6094 & 0.7835 & 0.8126 & 1.5 \\ \hline
            $\epsilon_1=1+1i, \epsilon_2=0$ & 0.1397 & 0.1774 & 0.6032 & 0.7043 & 1.4142 \\ \hline
            $\epsilon_1=0, \epsilon_2=0.5$ & 0.2272 & 0.5 & 0.5 & 0.5 & 0.5 \\ \hline
        \end{tabular}
        \caption{Unstructured and various structured stability radii for Example~\ref{exam:rev1}.}
        \label{tab:placeholder}
    \end{table}
\end{example}
}
\begin{example}\label{thm:unst}
    In this example, we generate random skew-Hermitian matrices $J\in\C^{n,n}$, Hermitian positive semidefinite matrices $R\in \C^{n,n}$, and Hermitian positive definite matrices $Q\in \C^{n,n}$, for different values of $n$, and record the various stability radii results  for the DH system $\dot x = (J-R)Qx$ in Table~\ref{tab:example1}.
    \begin{table}[h!]
        \centering
        \begin{tabular}{|c|c|c|c|c|} \hline
            $n$ & $r(J,R)$ & $r^S(J,R)$ & $r^{S_i}(J,R)$ & $r^{S_d}(J,R)$ 
            \\ \hline
            3 & 0.1501 & 0.2074 & 0.5779 & 0.6917 
            \\ \hline
            4 & 0.1309 & 0.1813 & 0.6413 & 0.7863 
            \\ \hline
            5 & 0.1959 & 0.2673 & 0.3773 & 0.4450 
            \\ \hline
            6 & 0.1469 & 0.1893 & 0.6019 & 0.7281 
            \\ \hline
            7 & 0.2911 & 0.3636 & 0.6234 & 0.7001 
            \\ \hline
            8 & 0.1571 & 0.2008 & 0.5623 & 0.6405 
            \\ \hline
            9 & 0.3759 & 0.4544 & 0.8946 & 1.0228 
            \\ \hline
        \end{tabular}
        \caption{Unstructured and various structured stability radii while perturbing $J$ and $R$.}
        \label{tab:example1}
    \end{table}
    The second column displays the unstructured stability radius $r(J,R)$ obtained in Theorem~\ref{thm:unstr}. The third and fourth column display the structured stability radii $r^S(J,R)$ and $r^{S_i}(J,R)$ obtained in Theorems~\ref{thm:Herm} and~\ref{thm:idef}, respectively. Finally, the fifth column records the structured stability radius $r^{S_d}(J,R)$ obtained in Theorem~\ref{cor:def1}. 
    Table~\ref{tab:example1} illustrates that the structured stability radii  are larger than the unstructured ones implying that {under structured perturbations the system is much more robustly stable.}
\end{example}

\begin{example}
    In this example we consider a DH system that has the structure of that in the finite element analysis of disk brake squeal from Example~\ref{example:brakesqueal} in a first-order formulation given by $\dot x = (J-R)Qx$, where
    \begin{equation}\label{disk}
        J=\mat{cc} G & (K+\frac{1}{2}N)\\ -(K+\frac{1}{2}N^*) & 0\rix, R= \mat{cc} D & \frac{1}{2}N \\ \frac{1}{2}N^* & 0 \rix, Q= \mat{cc} M & 0 \\ 0 &K \rix^{-1}.
    \end{equation}
    For $N=0$ the system $\dot x = (J-R)Qx$ is DH and hence asymptotically stable; i.e., all eigenvalues are in the open left half of the complex plane. 
    
    As test cases we randomly generate matrices $G, M, K$, and $D \in \C^{m,m}$ with $G$  skew-Hermitian, $M, K, D$ Hermitian positive semidefinite, and $N=0$, for different  $m\in\{50,60,70,80,90,100\}$. We record the different structured stability radii  in Table~\ref{tab:example2}.
    \begin{table}[ht!]
        \centering
        \begin{tabular}{|c|c|c|c|} \hline
             DH system size & $r^S(J,R)$ & $r^{S_i}(J,R)$ & $r^{S_d}(J,R)$ 
             \\
               & Theorem~\ref{thm:Herm} & Theorem~\ref{thm:idef} & Theorem~\ref{cor:def1} 
               \\ \hline
             100 & 0.0094 & 0.8027 & 1.6027 
             \\ \hline 
             120 & 0.0072 & 0.8742 & 2.0765 
             \\ \hline
             140 & 0.0074 & 0.7424 & 2.2910 
             \\ \hline
             160 & 0.0072 & 0.7341 & 1.5817 
             \\ \hline
             180 & 0.0045 & 0.7874 & 1.9112 
             \\ \hline
             200 & 0.0047 & 0.7884 & 1.7741 
             \\ \hline
        \end{tabular}
        \caption{Structured stability radii of the DH system as in \eqref{disk} with $N=0$.}
        \label{tab:example2}
    \end{table}
    Since here $R$ is Hermitian positive semidefinite and singular, by Theorem~\ref{thm:idef}, the third column of Table~\ref{tab:example2} gives a lower bound to the structured stability radius $r^{S_i}(J,R)$, which is also a lower bound to the structured stability radius $r^{S_d}(J,R)$. 
    
 Note that for large $m$ the computation times for computing the structured stability radii becomes prohibitively large, so for realistic large scale problems the optimization problem should be combined with model reduction as in 
    \cite{AliMM20,GraMQSV16}.
\end{example}

In the following example we consider the  case that we first solve a Lyapunov inequality to make the dissipative part optimal.

\begin{example}
    Consider a randomly generated asymptotically stable matrix $A\in \C^{5,5}$. We first compute a solution $\hat X > 0$ satisfying the equation
\[
-A^H \hat X - \hat X A = -2\mu\hat X,
\]
where $\mu$ is the spectral abscissa of $A$. We do this by solving the Lyapunov inequality
\[
{(A-(\mu + \epsilon)I)}^HX + X(A-(\mu + \epsilon)I) < - \epsilon I,
\]
with $\epsilon = 10^{-10}$, using the \emph{YALMIP} toolbox in MATLAB. With the solution $\hat X$ we construct $\hat J$ and $\hat R$ as described in Remark~\ref{rem:optimal}. The unstructured stability radius as in Theorem~\ref{thm:unstr} corresponding to this decomposition $\hat A=\hat J - \hat R$ was computed as $2.0774$. We then generate additional solutions $\tilde X > 0$ by solving
\begin{equation}\label{lyaz}
A^HX + XA = -Z^HZ,
\end{equation}
for randomly chosen matrices $Z$, using the MATLAB function \texttt{lyap} and observe that the resulting stability radii of the matrix $\tilde A = \tilde J-\tilde R$ corresponding to all such $\tilde X$ were consistently smaller than 2.0774. This suggests that the decomposition $\hat A=\hat J - \hat R$ corresponding to $\hat X$ yields a very robust DH representation.

We also computed the structured stability radii $r^{S_d}, r^{S_i},$ and $r^S$  for the matrix $\hat A = \hat J - \hat R$ corresponding to $\hat X$, and observed that these values were equal to the spectral abscissa $\mu$, i.e.,
\[
r^{S_d}(\hat J,\hat R)= r^{S_i}(\hat J, \hat R)= r^S(\hat J, \hat R) = \mu.
\]
Using other solutions $\tilde X$ of the Lyapunov equation~\eqref{lyaz}, the structured stability radii $r^{S_d}$ and $r^{S_i}$ were consistently larger than $\mu$. 
This suggests a potential direction for future research, i.e. to maximize  the structured stability radii over different decompositions $A=J-R$.
\end{example}

\section{Conclusions}
We have derived explicit computable formulas for various stability radii of DH systems with respect to skew-Hermitian perturbations in $J$ and three different types of structured perturbations in $R$, first, Hermitian negative semidefinite perturbations, which keep the perturbed matrix $\tilde R$ positive semidefinite; second, Hermitian perturbations such that the perturbed matrix $\tilde R$ is Hermitian positive semidefinite; and third, Hermitian perturbations. The results demonstrate that limiting perturbations to those that preserve structure results in larger robustness, as it takes substantially larger perturbations to shift an eigenvalue to the imaginary axis. The case where perturbations in all three matrices $J$, $R$, and $Q$ are considered while computing structured stability radii is still an open problem and left for future work.

{
The computational method employed in Section~\ref{sec:numerical} for calculating the different stability radii in general cannot be guaranteed  to converge to a global optimum, as noted in Remark~\ref{rem:place1}. The development of globally convergent algorithms remains a significant area for future research.
}

\section*{Author contributions}%
\addcontentsline{toc}{section}{Author contributions}
Anshul Prajapati performed writing - original draft and conceptualization. Peter Benner performed writing - review and editing, conceptualization, and funding acquisition. Volker Mehrmann and Punit Sharma performed writing - review and editing, and conceptualization.


\section*{Acknowledgments}%
\addcontentsline{toc}{section}{Acknowledgments}
Anshul Prajapati acknowledges the Max Planck Institute for support through a postdoctoral fellowship. Punit Sharma acknowledges the support of the SERB-CRG grant (CRG/2023/003221) and SERB-MATRICS grant by Government of India. 



    \addcontentsline{toc}{section}{REFERENCES}
\bibliographystyle{plainurl}
\bibliography{exampleref}

\end{document}